\documentclass{amsart}
\usepackage{graphicx} 
\usepackage{dsfont,xcolor, hyperref}
\usepackage{enumerate}
\usepackage[nodayofweek]{datetime}

\definecolor{my-blue}{cmyk}{1,0.6,0,0}

\definecolor{my-green}{cmyk}{0.8,0,1,0.5}

\hypersetup{
colorlinks=true, 
linkcolor=my-blue, 
citecolor=my-green,
urlcolor=black
}

\newcommand\FF{{\mathbb F}}
\newcommand\ZZ{{\mathbb Z}}
\newcommand{\Fq}{{\FF_{\!q}}}

\newcommand{\cI}{\mathcal{I}}
\newcommand{\bK}{\mathbb{K}}

\newcommand{\one}{\mathds{1}}

\newcommand{\tr}{\mathrm{tr}} 

\newcommand{\degtau}{\deg_\tau} 
\newcommand{\degsigma}{\deg_\sigma}
\newcommand{\ord}{\mathrm{ord}}
\newcommand{\ordsigma}{\ord_\sigma}

\newcommand{\mot}{\mathsf{M}}  
\newcommand{\mothat}{\hat{\mot}^{\sigma}} 
\newcommand{\carl}{\mathsf{C}}  

\newcommand{\ps}[1]{[\![#1]\!]}  
\newcommand{\ls}[1]{(\!(#1)\!)}
\renewcommand{\sp}[1]{\{#1\}} 
\newcommand{\sps}[1]{\{\!\{#1\}\!\}} 
\newcommand{\sls}[1]{(\!\{#1\}\!)}

\newcommand{\partdef}[1]{ \left\{ \begin{array}{ll} #1 \end{array} \right. }

\newcommand{\vect}[1]{\text{\boldmath $#1$\unboldmath}} 
\newcommand{\svect}[2]{\left( \begin{matrix} #1_{1}\\ \vdots \\ #1_{#2}\end{matrix}\right)}

\newcommand{\basis}[2]{\{ #1_{1},\ldots, #1_{#2}\}}

\newenvironment{smatrix}{\left(\begin{smallmatrix}}{\end{smallmatrix}\right)}

\newcommand{\Diagonal}[2]{\begin{pmatrix} 
#1 & 0 & \cdots & 0\\ 
0 & \ddots & \ddots & \vdots\\
\vdots & \ddots & \ddots & 0 \\
0 & \cdots & 0 & #2 \end{pmatrix}
}

\DeclareMathOperator{\Mat}{Mat}
\DeclareMathOperator{\GL}{GL}
\DeclareMathOperator{\U}{U}
\DeclareMathOperator{\diag}{diag}

\theoremstyle{plain}
\newtheorem{thm}{Theorem}[section]
\newtheorem{cor}[thm]{Corollary}
\newtheorem{lem}[thm]{Lemma}
\newtheorem{prop}[thm]{Proposition}
\newtheorem{conj}[thm]{Conjecture}
\newtheorem{mainthm}{Theorem}

\theoremstyle{definition}
\newtheorem{defn}[thm]{Definition}
\newtheorem{exmp}[thm]{Example}
\newtheorem{rem}[thm]{Remark}
\newtheorem*{ack}{Acknowledgments}

\title{On purity of Anderson $t$-modules}
\author{Yen-Tsung Chen and Andreas Maurischat}
\newdate{date}{28}{07}{2026}
\date{\displaydate{date}}

\begin{document}

\begin{abstract}
	In the theory of abelian Anderson $t$-modules, pure Anderson $t$-modules play an important role.
	Namoijam and Papanikolas also introduced the notions of \emph{strictly pure} and \emph{almost strictly pure}
	$t$-modules which are special classes of pure $t$-modules.
	Whereas it is well known that not all pure $t$-modules are isomorphic to strictly pure ones (e.g.~the tensor powers of the Carlitz module), the same question for 
	almost strictly pure $t$-modules was not answered yet.
	In this article, we show that every pure Anderson $t$-module defined over a field $K$ is indeed isomorphic to an almost strictly pure 
	Anderson $t$-module after base change to a suitable finite algebraic extension of $K$. We also give a necessary and sufficient criterion when such an isomorphism to a strictly pure Anderson $t$-module is possible.
\end{abstract}



\maketitle

\setcounter{tocdepth}{1}
\tableofcontents

\section{Introduction}\label{sec:introduction}

In the classical theory of mixed motives over number fields, the pure motives are the building blocks in the sense that every mixed motive has a weight filtration whose factors are pure motives.
In the function field setting, the role of motives is played by abelian Anderson $t$-motives introduced by Anderson \cite{ga:tm}. Here, the pure $t$-motives still play a central role, although there are simple $t$-motives that are not pure (see e.g.~\cite[Example 6.3]{am:aefam}).

In \cite{cn-mp:hpqpam}, Namoijam and Papanikolas introduced other notions of purity for Anderson $t$-motives which they call \emph{strictly pure} and \emph{almost strictly pure} Anderson $t$-motives.\footnote{Actually, they defined these terms for Anderson $t$-modules, but in the introduction, we phrase them for $t$-motives, due to the anti-equivalence between abelian Anderson $t$-modules and abelian Anderson $t$-motives given in \cite[Theorem 1]{ga:tm}. More details are given in Section \ref{sec:anderson-t-modules}.} 
These $t$-motives are indeed special classes of pure $t$-motives (see \cite[Remark~5.5.5]{dg:bsffa} and \cite[Theorem~7.3.6]{dt:ffa} for strictly pure case, and see \cite[\S 2.5.2, p.~112]{uh-akj:pthshcff} and \cite[Section 4.5, p.~55]{cn-mp:hpqpam} for almost strictly pure case), and many prominent examples of pure $t$-motives, like the $t$-motive of Drinfeld modules, Carlitz tensor powers etc., are indeed almost strictly pure when considered in their standard representation (see \cite[Remark 4.90]{cn-mp:hpqpam}).

\medskip

Let us explain the issue in more detail.

We fix a finite field $\Fq$, and an arbitrary extension field $K$ of $\Fq$. We will consider
the skew polynomial ring $K\sp{\tau,t}$ in two indeterminants $t$, $\tau$,
\[  K\sp{\tau,t}=\left\{ \sum_{i=0}^n\sum_{j=0}^m \alpha_{ij}\tau^i t^j \,\middle|\, n,m\geq 0, \alpha_{ij}\in K\right\} \]
with multiplication uniquely given by additivity and the rules
\[ \tau \cdot \alpha = \alpha^q\cdot \tau,\quad t\cdot \alpha=\alpha\cdot t \]
for all $\alpha\in K$, as well as $\tau\cdot t=t\cdot \tau$.

This skew polynomial ring contains the skew polynomial rings $K\sp{\tau}$ and $K\sp{t}$. As $t$ commutes with all elements in $K$ by definition, the latter is the usual commutative polynomial ring, and we will write $K[t]$ for it.

An abelian Anderson $t$-motive over $K$ is a left $K\sp{\tau,t}$-module $\mot$ which is free of finite rank as left module over $K\sp{\tau}$, and free of finite rank as module over $K[t]$, and which satisfies a certain additional condition that will not be relevant here.

The Anderson $t$-motive $\mot$ is said to be \emph{pure} if there exist positive numbers $u,v\in \ZZ$, and a $K\ps{\frac{1}{t}}$-lattice $\tilde{\Lambda}$ in the (finite dimensional) $K\ls{\frac{1}{t}}$-vector space $K\ls{\frac{1}{t}}\otimes_{K[t]} \mot$, such that the $K\ps{\frac{1}{t}}$-spans of $t^u\tilde{\Lambda}$ and $\tau^v\tilde{\Lambda}$ are equal. The ratio $\frac{u}{v}$ is then called the \emph{weight} of $\mot$.

On the other hand, if $\mot$ is given with a fixed $K\sp{\tau}$-basis, we identify the $t$-motive with the row vectors over $K\sp{\tau}$,
\[\mot\cong K\sp{\tau}^{1\times d}, \]
and we have
\[   t\cdot (x_1,\ldots, x_d) = (x_1,\ldots, x_d)\cdot D \]
for some matrix $D\in \Mat_{d\times d}(K\sp{\tau})$ which we write as $D=D_0+D_1\tau+\ldots +D_s\tau^s$ with $D_j\in \Mat_{d\times d}(K)$ for all $j=0,\ldots, s$, and $D_s\ne 0$.

The $t$-motive $\mot$ (with this fixed choice of basis) is called \emph{strictly pure}, if the top coefficient matrix $D_s$ is invertible.
It is called \emph{almost strictly pure}, if for some $n\geq 1$, the top coefficient matrix of the $n$-th power $D^n$ is invertible.

Of course, strictly pure $t$-motives are almost strictly pure, and as mentioned above, almost strictly pure $t$-motives are pure.
That raises the questions whether the reverse implications also hold.

As an example that almost strictly pure $t$-motives exist which are not strictly pure, consider
 the second Carlitz tensor power $\carl^{\otimes 2}\cong K\sp{\tau}^{1\times 2}$ which is given by
\[  t\cdot (x_1,x_2) = (x_1,x_2)\cdot \begin{pmatrix} \theta & 1 \\ \tau & \theta
\end{pmatrix}, \]
for a certain $\theta\in K$, i.e., 
\[ D=\begin{pmatrix} \theta & 1 \\ \tau & \theta
\end{pmatrix} = \begin{pmatrix} \theta & 1 \\ 0 & \theta
\end{pmatrix} + \begin{pmatrix} 0 & 0 \\ 1 & 0
\end{pmatrix}\cdot \tau. \]
As the top coefficient matrix is not invertible, $\carl^{\otimes 2}$ is not strictly pure (with this choice of $K\sp{\tau}$-basis).\footnote{It can be shown that no choice of $K\sp{\tau}$-basis would make $\carl^{\otimes 2}$ strictly pure.} However,
\[ D^2 = \begin{pmatrix} \theta & 1 \\ \tau & \theta
\end{pmatrix}\cdot \begin{pmatrix} \theta & 1 \\ \tau & \theta
\end{pmatrix} = \begin{pmatrix} \theta^2+\tau & 2\theta \\ (\theta^q+\theta)\tau & \tau+\theta^2
\end{pmatrix} =\begin{pmatrix} \theta^2 & 2\theta \\ 0 & \theta^2
\end{pmatrix}+\begin{pmatrix} 1 & 0 \\ \theta^q+\theta & 1
\end{pmatrix}\cdot \tau,
\]
has the top coefficient matrix $\begin{pmatrix} 1 & 0 \\ \theta^q+\theta & 1
\end{pmatrix}$ which is invertible. So $\carl^{\otimes 2}$ is almost strictly pure.
So, it is an example of an almost strictly pure $t$-motive that is not strictly pure.

Lucas \cite{al:paspatm} showed that there are $t$-motives with fixed $K\sp{\tau}$-bases that are pure but not almost strictly pure. For each dimension $d\geq 2$, he provided such an example.

In this article, we investigate the case when we don't fix a $K\sp{\tau}$-basis aforehand.

\subsection*{Main results}

Our first main result implies that over algebraically closed fields, the terms \emph{pure} and \emph{almost strictly pure with respect to some basis} are indeed equivalent.

\begin{mainthm} (see Theorem \ref{thm:pure-implies-almost-strictly-pure}) \label{mainthm:A}
Let $\bK$ be an algebraically closed field containing $\Fq$, and let $\mot$ be a pure Anderson $t$-motive over $\bK$ of $\bK\sp{\tau}$-rank $d$.
Then there exists a $\bK\sp{\tau}$-basis $\{\kappa_1,\ldots, \kappa_d\}$ of $\mot$, so that $\mot$ is almost strictly pure with respect to $\{\kappa_1,\ldots, \kappa_d\}$. In addition, the $\bK\{\tau\}$-basis can be chosen to which $\mot$ is strictly pure, if and only if the reciprocal of the weight of $\mot$ is an integer.
\end{mainthm}

For an arbitrary extension field $K$ of $\Fq$, we conclude the following.

\begin{cor} (see Corollary \ref{cor:almost-strictly-pure-over-finite-extension})
     Let $\mot$ be a pure Anderson $t$-motive over $K$ of $K\sp{\tau}$-rank $d$.
Then there exists a finite field extension $L$ of $K$, and an 
$L\sp{\tau}$-basis $\{\kappa_1,\ldots, \kappa_d\}$ of $\mot_L:=L\sp{\tau}\otimes_{K\sp{\tau}}\mot$, so that $\mot_L$ is almost strictly pure with respect to $\{\kappa_1,\ldots, \kappa_d\}$. In addition, the $L\{\tau\}$-basis can be chosen to which $\mot_L$ is strictly pure, if and only if the reciprocal of the weight of $\mot$ is an integer.
\end{cor}

Using Anderson's anti-equivalence of categories, the corresponding statement for abelian Anderson $t$-modules is the following.

\begin{mainthm} (see Theorem \ref{thm:almost-strictly-pure-over-extension}) \label{mainthm:B}
   Let $K$ be a field containing $\Fq$, and let $\phi$ be a pure Anderson $t$-module over $K$ of dimension $d$ and weight $w$.
Then there exists a finite extension $L$ of $K$, and an almost strictly pure $t$-module $\psi$ over $L$ such that
the base extension of $\phi$ to $L$ is isomorphic to $\psi$. 
The $t$-module $\psi$ over $L$ can be chosen to be strictly pure, if and only if $\frac{1}{w}\in\mathbb{Z}$.
\end{mainthm}

\begin{rem}
    We would like to mention that one could also prove our theorems using the $t$-comotive of the $t$-modules, as there is a duality between the $t$-motive and the $t$-comotive (see e.g.~\cite[Theorem 2.5.13]{uh-akj:pthshcff} or \cite[Theorem 1.2]{qg-am:pamfrfe}). However, using the $t$-motive is more comfortable as it has a natural left $K\sp{\tau}$-module structure.
\end{rem}

\subsection*{Applications}
    In what follows, we present two applications of our main results. Given two pure $t$-motives $\mot_1$ and $\mot_2$ over $K$, it is known that their tensor product over $K[t]$, given by $\mot:=\mot_1\otimes_{K[t]}\mot_2$, is again a pure $t$-motive. In addition, the weight of $\mot$ is the sum of the weights of $\mot_1$ and $\mot_2$. According to some explicit calculations, it has been shown that the tensor product of the $t$-motives of two Drinfeld modules defined over $K$ is almost strictly pure with respect to some canonical choices of the $K\{\tau\}$-basis. It is natural to ask if we can still have the almost strict purity when replacing the $t$-motives of Drinfeld modules by arbitrary $t$-motives of almost strictly pure $t$-modules. The first application of our main result gives an affirmative answer to this question for algebraically closed fields $\bK$.

    \begin{mainthm}(see Theorem \ref{thm:tensor-products}) \label{mainthm:C}
        Let $\phi$ and $\psi$ be two almost strictly pure $t$-modules defined over $\bK$ of weight $w_\phi$ and $w_\psi$ respectively. Let $\mot:=\mot_\phi\otimes_{\bK[t]}\mot_\psi$ be the tensor product of their $t$-motives. Then there is an almost strictly pure $t$-module $\rho$ defined over $\bK$ such that
        \[
            \mot\cong\mot_\rho.
        \]
        In other words, $\mot_\phi\otimes_{\bK[t]}\mot_\psi$ is almost strictly pure with respect to an appropriate $\bK\{\tau\}$-basis. If $1/w_\phi+1/w_\psi\in\mathbb{Z}$, then the $\bK\{\tau\}$-basis can be chosen to which $\mot_\phi\otimes_{\bK[t]}\mot_\psi$ is strictly pure.
    \end{mainthm}

    For the second application, recall that for a given Drinfeld module $\phi$ defined over a global function field $K$, we denote by $\phi(K)$ the $\mathbb{F}_q[t]$-module whose underlying space is given by $K$ and the module structure comes from  $\phi$. It was known from the remark at the beginning of the proof of Theorem 5 in \cite{Denis+1992+285+302} that $\phi(K)$ is never finitely generated. Thus, a direct analog of the Mordell-Weil theorem fails in this setting. Nevertheless, Poonen \cite{bp:lhfmwtdm} proved that $\phi(K)$ is isomorphic to a free $\mathbb{F}_q[t]$-module of countably infinite rank $\aleph_0$ with a finite torsion module. Inspired by this result, a $t$-module $\phi$ defined over a global function field $K$ is said to be satisfying Poonen's Mordell-Weil theorem if the $\mathbb{F}_q[t]$-module $\phi(K)$ has the same structure as in the case of Drinfeld modules. Using the theory of local height functions, Kuan \cite{Kuan2022-kh} successfully proved that every almost strictly pure $t$-module defined over a global function field satisfies Poonen's Mordell-Weil theorem. As an application of our main result, we can further extend the result of Kuan to any pure $t$-module after a finite base extension.

    \begin{mainthm}(see Theorem \ref{thm:MordellWeil})
        Let $\phi$ be a pure $t$-module defined over a global function field $K$. Then there is a finite extension $L$ over $K$ such that the base extension of $\phi$ to $L$ satisfies Poonen's Mordell-Weil theorem, that is, the $\Fq[t]$-module $\phi^L(L)$ is the direct sum of its finite torsion submodule with a free $\Fq[t]$-module of rank $\aleph_0$.
    \end{mainthm}

\subsection*{Outline of the paper}

In Section \ref{sec:notation},
we introduce the basic notation on skew polynomial rings, skew power series rings and skew Laurent series rings that we use in this article.
Section \ref{sec:matrices} is devoted to structural theorems on square matrices over such skew Laurent series rings, and we prove a non-commutative Birkhoff factorization for such matrices (Theorem \ref{thm:matrix-decomposition}) as well as a theorem on special representatives in conjugacy classes (Theorem \ref{thm:conjugation}). These two theorems will be the computational essence in proving our main theorems.

In Section \ref{sec:anderson-t-modules}, we introduce abelian Anderson $t$-modules and their $t$-motives, and we give the definitions of the notion \emph{pure}, \emph{strictly pure} and \emph{almost strictly pure} for $t$-modules and $t$-motives, and we discuss the relations between them.

Our main theorems \ref{mainthm:A} and \ref{mainthm:B} are proven in Section \ref{sec:main-theorem}. A preparatory step for the proof is an observation of the second author in \cite{am:aefam} which allows to rephrase purity of a $t$-motive $\mot$ in terms of $K\sps{\tau^{-1}}$-lattices inside $K\sls{\tau^{-1}}\otimes_{K\sp{\tau}}\mot$ (see Proposition \ref{prop:iso-of-completions} and Corollary \ref{cor:stable-lambda}).

The applications of our result that we mentioned above are given in Section \ref{sec:applications}.
We conclude our paper with two examples of pure $t$-modules in Section \ref{sec:examples}. Here, we explicitly determine almost strictly pure $t$-modules that are isomorphic to the given ones by applying the algorithms that are provided by the proofs of the main theorems.

\begin{ack}
    The authors would like to thank D.~Thakur and \"O.~\"Ulkem for their helpful comments on a previous version of this paper.
\end{ack}

\section{Notation}\label{sec:notation}

Let $\Fq$ be the field with $q$ elements, $K$ a field containing $\Fq$, and $\bK$ an algebraic closure of $K$.

$\tau:\bK\to \bK,x\mapsto x^q$ is the Frobenius automorphism, and $\sigma=\tau^{-1}$ its inverse.

We use the skew polynomial ring 
\[  \bK\sp{\tau}=\left\{ \sum_{i=0}^n \alpha_{ij}\tau^i \,\middle|\, n\geq 0, \alpha_{i}\in \bK\right\} \]
with multiplication uniquely given by additivity and the rule
$\tau \cdot \alpha = \alpha^q\cdot \tau$
for all $\alpha\in \bK$, the skew power series ring 
\[ \bK\sps{\sigma}=\left\{ \sum_{i=0}^\infty \alpha_i\sigma^i \,\middle|\, \alpha_i\in \bK\right\}, \] 
and the skew Laurent series ring 
\[ \bK\sls{\sigma}=\left\{ \sum_{i=i_0}^\infty \alpha_i\sigma^i \,\middle|\, i_0\in \ZZ, \alpha_i\in \bK\right\}, \]
with $\sigma \cdot \alpha = \alpha^{1/q}\cdot \sigma$ for all $\alpha\in \bK$. 
Of course, $\bK\sls{\sigma}$ contains $\bK\sps{\sigma}$. Further, the ring $\bK\sp{\tau}$ is naturally embedded into $\bK\sls{\sigma}$ via
\[ \sum_{i=0}^n \alpha_i\tau^i \mapsto \sum_{i=0}^n \alpha_i\sigma^{-i}. \]

For $f\in \bK\sls{\sigma}$, $n\in \ZZ$, we define the \emph{$n$-th twist} of $f$ as $f^{(n)}=\sigma^{-n}\cdot f\cdot \sigma^n\in \bK\sls{\sigma}$, i.e.,~
writing $f=\sum_{l=k}^\infty \alpha_l\sigma^l$, we have
\begin{equation}\label{eq:n-th twist}
     f^{(n)} = \sum_{l=k}^\infty \alpha_l^{q^n}\sigma^l .
\end{equation}

For $0\ne f=\sum_{l=k}^\infty \alpha_l\sigma^l\in \bK\sls{\sigma}$, and $w\in \ZZ$, we define the \emph{$\sigma$-order of $f$} and the \emph{$w$-th coefficient of $f$} as
\[  \ordsigma(f)=\inf \{ l \mid \alpha_l\ne 0\} \in \ZZ,\quad \text{and}\quad  c_w(f)=\alpha_w,  \]
respectively.
We will also use the \emph{$\sigma$-degree of $f$},
\[  \degsigma(f)=\sup \{ l \mid \alpha_{l}\ne 0\} \in \ZZ \cup \{\infty\}, \]
but actually only if $f\in \bK\sp{\tau,\sigma}$, i.e.~the supremum is strictly smaller than $\infty$.

For $f=0$, we will set $\ordsigma(f)=\infty$, $\degsigma(f)=-\infty$, and $c_w(f)=0$ for all $w\in \ZZ$.

We remark, that $f\in \bK\sls{\sigma}$ belongs to $\bK\sp{\tau}$, if and only if $\degsigma(f)\leq 0$, in which case
\[ f=\sum_{l=k}^0 \alpha_l\sigma^l=\sum_{m=0}^{-k} \alpha_{-m}\tau^m\in \bK\sp{\tau}. \]
We denote its \emph{degree as a polynomial in $\tau$} by $\degtau(f)$, and observe that 
\[\degtau(f)=-\ordsigma(f).\]

Observe also that by the explicit formula \eqref{eq:n-th twist} for the twist, it is clear that $\bK\sp{\tau}$ and $\bK\sps{\sigma}$ are stabilized by the twist, and that $\ordsigma$, $\degsigma$ and $\degtau$ are invariant under twisting.

We extend all these notation to matrices $D\in \Mat(\bK\sls{\sigma})$, by setting
\begin{align*}
    D^{(n)} &=(D_{ij}^{(n)})_{ij}\in \Mat(\bK\sls{\sigma}),\\
     c_w(D) &=(c_w(D_{ij}))_{ij}\in \Mat(\bK) , \\
     \ordsigma(D) &= \min_{i,j} \ordsigma(D_{ij}),\\
     \degsigma(D)&= \max_{i,j} \degsigma(D_{ij}),\quad \text{and}\\
     \degtau(D)&= \max_{i,j} \degtau(D_{ij}).
\end{align*}

Finally, we introduce the following convention for bases of free modules. Let $R$ be a ring and $M$ be an $R$-module that is free of rank $d<\infty$. Let $\{b_1,\dots,b_d\}$ be any $R$-basis of $M$. We introduce the abbreviation
\[
    \vect{b}=\svect{b}{d}\in\Mat_{d\times 1}(M)
\]
for these tuples of elements of $M$, and say that such a tuple is a basis, meaning that the elements of the tuple build a basis.

\section{Matrices over skew Laurent series rings}\label{sec:matrices}

In this section, we prove two structural theorems on matrices in $\GL_d(\bK\sls{\sigma})$. The first one is 
a non-commutative version of the Birkhoff factorization, whereas the second one is about special elements in conjugacy classes.

\begin{thm}\label{thm:matrix-decomposition}
    Let $C\in \GL_d(\bK\sls{\sigma})$. Then there exist $A\in \GL_d(\bK\sps{\sigma})$, $B\in \GL_d(\bK\sp{\tau})$, and a diagonal matrix $D=\diag(\sigma^{k_1},\ldots, \sigma^{k_d})$ with integers $k_1,\ldots,k_d$ such that
    \[ C= A\cdot D \cdot B .\]    
    Moreover, the matrices can be chosen so that $k_1\leq k_2\leq \ldots \leq k_d$. The matrices can also be chosen so that $k_1\geq k_2\geq \ldots \geq k_d$.
\end{thm}

\begin{rem}
The theorem also holds for perfect fields $\bK$ that are not algebraically closed. This will be obvious from the proof.
As we only need the theorem for algebraically closed $\bK$, we stated it this way to not overload notation.
\end{rem}

\begin{proof}[Proof of Theorem \ref{thm:matrix-decomposition}]
    We will show that we can transform $C$ into such a diagonal matrix $D$ by performing elementary $\bK\sp{\tau}$-column operations and elementary $\bK\sps{\sigma}$-row operations.

    The column operations correspond to multiplication from the right with matrices in $\GL_d(\bK\sp{\tau})$, and the row operations correspond to multiplication from the left with matrices in $\GL_d(\bK\sps{\sigma})$.

    Be also aware that for the column operations the scalars have to be multiplied from the right.

\textbf{Step 1: Turn $C$ into a lower triangular matrix with powers of $\sigma$ on the diagonal:}

Firstly, in the last column, choose the index $i_0$ such that the $\sigma$-order of the $i_0$-th entry is minimal among the $\sigma$-orders of that column and switch the $i_0$-th row of the matrix with the last row. Then all other entries in the last column are $\bK\sps{\sigma}$-multiples of the last entry.
Therefore, we can subtract suitable $\bK\sps{\sigma}$-multiples of the last row from the other rows such that the non-diagonal entries of the last column become $0$. After rescaling the last row by the suitable factor in $\bK\sps{\sigma}^\times$, we can achieve that the $(d,d)$-entry becomes a power of $\sigma$ (equal to the $\sigma$-order of the former $(d,d)$-th entry). 
Hence, we achieved that the last column has the desired form.

We repeat this process successively for $j=d-1$ to $1$ with the upper left $(j\times j)$-submatrix. So afterwards the matrix $C$ has been transformed into lower triangular form with diagonal entries $\sigma^{k_1},\ldots, \sigma^{k_d}$ for integers $k_1,\ldots,k_d$.

\medskip

\textbf{Step 2: Inductively for $i=2$ to $d$ achieve that the off-diagonal entries in the $i$-th row are $0$:}

For fixed $i$, assume that we achieved this step for all the previous rows.
For achieving that also the off-diagonal entries in the $i$-th row become $0$, we repeat the following three tasks until the if-condition in task \eqref{item:step-2} is not fulfilled anymore.
        \begin{enumerate}
            \item \label{item:step-1} For all $j<i$ for which the $(i,j)$-th entry $C_{ij}=\sum_{l=k}^\infty \alpha_l\sigma^l$ has $\sigma$-order $k<k_i$, let
            $f=\sigma^{-k_i}\cdot \sum_{l=k}^{k_i} \alpha_l\sigma^{l}=\sum_{l=k}^{k_i} (\alpha_l)^{q^{k_i}}\tau^{k_i-l}\in \bK\sp{\tau}$, and
            subtract the $f$-multiple of the $i$-th column from the $j$-th column. After that the $\sigma$-order of 
            the diagonal entry $C_{ii}$ is strictly smaller than the entries $C_{ij}$ for $j<i$.
            	Since, the first $i-1$ entries in the $i$-th column are zero, this doesn't change the first $i-1$ rows.
            \item \label{item:step-2a} For all $j<i$ such that the $\sigma$-order of the $(i,j)$-th entry is larger than or equal to that of the $(j,j)$-th entry, subtract a suitable $\bK\sps{\sigma}$-multiple of the $j$-th row from the $i$-th row to kill the $(i,j)$-th entry.
            \item \label{item:step-2} If there is some $j<i$ with $C_{ij}\ne 0$, choose the largest such $j$. By the first two steps, the $\sigma$-order $l:=\ordsigma(C_{ij})$ of this $(i,j)$-th entry has to satisfy $k_j>l>k_i$, since $C_{jj}=\sigma^{k_j}$ and $C_{ii}=\sigma^{k_i}$, and we write 
            \[ C_{ij}=u\sigma^l\]
            with $u\in \bK\sps{\sigma}^\times$. Multiply $C$ from the left with the matrix
            \[ \begin{pmatrix}
                1 && && &&  && \\
                 &\ddots & && && && \\
                 && 1 && && && \\
                 &&  && && && \\
                &&& -\sigma^{k_j-l}u\sigma^{l-k_j} & & \sigma^{k_j-l} &&& \\
                &&&& \ddots &&&& \\
                &&& 0 & &u^{-1}  &&&   \\       
                 &&  && && && \\
                && &&  && 1&& \\
                && &&  && &\ddots& \\
                && &&  && && 1 
            \end{pmatrix} \in \GL_d(\bK\sps{\sigma}) \]
            differing from the identity matrix only in the entries $(j,j)$, $(j,i)$, and $(i,i)$, and afterwards swap the $j$-th and the $i$-th column. 
            The resulting matrix $\tilde{C}$ is again lower triangular. Namely, we have 
            \begin{align*}
            & \text{for }m<j: &\\
                \tilde{C}_{m,n}&=C_{m,n}=0, & \text{if } n>m; n\ne i,j\\
                \tilde{C}_{m,j}&=C_{m,i}=0 \text{ and }\tilde{C}_{m,i}=C_{m,j}=0, & \\[2mm]
            & \text{for }m=j: &\\
                \tilde{C}_{j,n}&=-\sigma^{k_j-l}u\sigma^{l-k_j}C_{j,n}+\sigma^{k_j-l} C_{i,n}=0+0=0      & \text{if } n>j, n\ne i\\
                \tilde{C}_{j,i}&=-\sigma^{k_j-l}u\sigma^{l-k_j}C_{j,j}+\sigma^{k_j-l} C_{i,j} &\\ &=-\sigma^{k_j-l}u\sigma^{l}+\sigma^{k_j-l}u\sigma^{l}=0, &\\[2mm]
            & \text{for }j<m<i: &\\
                \tilde{C}_{m,n}&=C_{m,n}=0, & \text{if } n>m, n\ne i\\
                \tilde{C}_{m,i}&=C_{m,j}=0, & \\[2mm]
            & \text{for }m=i: &\\
                \tilde{C}_{i,n}&=-u^{-1}C_{i,n}=0,     & \text{if } n>i\\[2mm]
            & \text{and for }m>i: &\\
                \tilde{C}_{m,n}&=C_{m,n}=0 .     & \text{if } n>m\\
            \end{align*}            
            In the upper $(i-1)$ rows only the entry $C_{jj}$ has changed to $\sigma^{k_j+(k_i-l)}$ and on the rest of the diagonal of the matrix only the $\sigma$-power of the $i$-th entry has changed: We now have $C_{ii}=\sigma^l$ with $l>k_i$. 
        \end{enumerate}
        The previous loop ends after finitely many steps, since in each loop the $\sigma$-order of $C_{ii}$ increases, but the $\sigma$-order of the diagonal entries $C_{jj}$ ($j<i$) stay the same or decrease.
        After this loop is finished, the off-diagonal entries in the $i$-th row are $0$. This finishes the induction step.

For achieving that the exponents $k_1,\ldots,k_d$ are non-decreasing or non-increasing, we only have to conjugate the obtained diagonal matrix by permutation matrices which permutes the diagonal entries accordingly.
\end{proof}

For the next theorem, we first need special subgroups of unipotent upper triangular matrices.

\begin{defn}
For $\vect{k}=(k_1,\ldots,k_d)\in \ZZ^d$ with $k_1\leq k_2\leq \dots \leq k_d$, we denote
\[  \U(\vect{k}) := \left\{ \begin{pmatrix}
        1 & g_{12} & \cdots & g_{1d}\\
        	& \ddots & \ddots & \vdots \\
         & & 1 & g_{d-1,d}\\
         & & & 1
    \end{pmatrix} \in \GL_d(\bK\sp{\sigma}) \,\, \middle| \,\, \degsigma(g_{ij})\leq k_j-k_i \right\}. \]
\end{defn}

\begin{lem}\label{lem:U-is-group}
The set $\U(\vect{k})$ is a subgroup of $\GL_d(\bK\sp{\sigma})$.
\end{lem}

\begin{proof}
Proving that $\U(\vect{k})$ is closed under multiplication is a direct computation, and we leave it as an exercise for the reader.
Since a unipotent upper triangular matrix $U$ can be written as $U=\one_d-T\in \GL_d(\bK\sp{\sigma})$ with a nilpotent matrix $T$, and $T^d=0$, the inverse of $U$ is explicitly given as
$U^{-1}=\one_d+T+T^2+\cdots + T^{d-1}$. One then easily checks that for $U\in \U(\vect{k})$, also $U^{-1}\in \U(\vect{k})$.
\end{proof}

\begin{thm}\label{thm:conjugation}
Let $\vect{k}=(k_1,\ldots,k_d)\in \ZZ^d$ with $k_1\leq k_2\leq \dots \leq k_d$, and
let $\Theta\in \GL_d(\bK\sls{\sigma})$ satisfy $v:=-\ordsigma(\Theta)>0$, as well as $c_{-v}(\Theta)\in \GL_d(\bK)$.
Then there exist a matrix $G\in \U(\vect{k})$ such that the entries of the matrix $H:=G^{-1}\Theta G$ satisfy
\[  \ordsigma(H_{ij})>-v+k_j-k_i \quad \forall 1\leq i <j\leq d. \]
\end{thm}

\begin{proof}
For $0\leq l\leq k_d-k_1$, we let
\[  \cI_l:=\{ (i,j) \mid 1\leq i<j\leq d, k_j-k_i\geq l \}. \]
Then for every matrix $S_{[l]}\in \Mat_d(\bK)$ with $(S_{[l]})_{ij}=0$ for all $(i,j)\not\in \cI_l$, the matrix
$U_{[l]}=\one_d-S_{[l]}\sigma^l$ belongs to $\U(\vect{k})$.

Inductively, for $l=0,\ldots ,k_d-k_1$, we will construct a matrix $\Theta_{[l+1]}\in \GL_d(\bK\sls{\sigma})$ as the conjugate
of a previously constructed $\Theta_{[l]}$ (respectively of $\Theta_{[0]}=\Theta$) by such a matrix $U_{[l]}=\one_d-S_{[l]}\sigma^l$, i.e.,
\[  \Theta_{[l+1]} = U_{[l]}\Theta_{[l]}(U_{[l]})^{-1} \]
 such that
\[  \ordsigma((\Theta_{[l+1]})_{ij})>-v+\min \{ l, k_j-k_i\} \]
 for all $(i,j)$ with $1\leq i<j\leq d$. 

Then
\[ G= U_{[k_d-k_1]}\cdot \dots \cdot U_{[1]}\cdot U_{[0]} \text{ and } H=\Theta_{[k_d-k_1+1]} \]
satisfy the properties given in the statement of the theorem, since $\U(\vect{k})$ is a group by Lemma \ref{lem:U-is-group}, and 
$k_d-k_1=\max\{ k_j-k_i \mid 1\leq i<j\leq d \}$.

\medskip

For the initial step $l=0$, we first set $\Theta_{[0]}=\Theta$, as mentioned above.

Since $c_{-v}(\Theta)\in \GL_d(\bK)$, by Lang isogeny there exists a matrix $F\in \GL_d(\bK)$ such that $F^{-1}\cdot F^{(v)} = c_{-v}(\Theta)$ (product of inverse and $v$-th twist).
Now, let $F=L\cdot U$ be the LU-decomposition of $F$ in $\GL_d(\bK)$, where the diagonal entries of $U$ all equal $1$.
We claim that $U_{[0]}=U$ fulfills the desired property. Indeed, $\cI_0=\{ (i,j) \mid 1\leq i<j\leq d\}$ and
as the diagonal entries of $U$ are $1$, $U$ is of the form $\one_d-S_{[0]}\sigma^0$ with $S_{[0]}$ strictly upper triangular.
Further, write $\Theta_{[0]}=\Theta=\sum_{l=-v}^\infty c_l(\Theta)\sigma^l$, then
    \begin{align*}
        \Theta_{[1]} &:= U_{[0]}\Theta (U_{[0]})^{-1} = U\cdot \sum_{l=-v}^\infty c_l(\Theta)\sigma^l U^{-1} \\
        &= \sum_{l=-v}^\infty U c_l(\Theta)(U^{-1})^{(-l)}\sigma^l 
	\end{align*}
	and hence        
    \begin{align*}
        c_{-v}(\Theta_{[1]}) &= U c_{-v}(\Theta)\cdot (U^{-1})^{(v)} = U\cdot  F^{-1}\cdot F^{(v)}\cdot (U^{-1})^{(v)}\\
        &= L^{-1}\cdot L^{(v)}
    \end{align*}
    is lower triangular. This means that for all $1\leq i<j\leq d$, we have $\ordsigma((\Theta_{[1]})_{ij})>-v+0$ as desired.

\medskip

For the induction step, assume that for given $l\in \{1,\dots, k_d-k_1\}$, the matrix $\Theta_{[l]}$ has been constructed satisfying
\begin{itemize}
\item $\ordsigma(\Theta_{[l]})=-v$,
\item $c_{-v}(\Theta_{[l]})\in \GL_d(\bK)$ is a lower triangular matrix,
\item $\ordsigma((\Theta_{[l]})_{ij})>-v+\min \{ l-1, k_j-k_i\} $ for all $(i,j)$ with $1\leq i<j\leq d$. 
\end{itemize}
For $l=1$, the above constructed matrix satisfies this hypotheses.

As explained above, we now seek to construct $\Theta_{[l+1]}$ as $\Theta_{[l+1]} = U_{[l]}\Theta_{[l]}(U_{[l]})^{-1}$ where
$U_{[l]}=\one_d-S_{[l]}\sigma^l$ for an appropriate matrix $S_{[l]}\in \Mat_d(\bK)$ with $(S_{[l]})_{ij}=0$ for all $(i,j)\not\in \cI_l$.

For any such matrix $S_{[l]}$, we have
\begin{align}
\Theta_{[l+1]} &= U_{[l]}\Theta_{[l]}(U_{[l]})^{-1} \notag \\
&= \left( \one_d-S_{[l]}\sigma^l\right)\Theta_{[l]} \left( \one_d+S_{[l]}\sigma^l + (S_{[l]}\sigma^l)^2 + \dots + (S_{[l]}\sigma^l)^{d-1} \right) \label{eq:theta_l+1} \\
&\equiv \Theta_{[l]} \mod{\sigma^{-v+l-1}}\notag
\end{align}  

So the first two bullet points of the induction hypothesis are also satisfied for $\Theta_{[l+1]}$, and we already have
$\ordsigma((\Theta_{[l+1]})_{ij})>-v+\min \{ l-1, k_j-k_i\} $ for all $(i,j)$ with $1\leq i<j\leq d$. 

Hence, we are done, if we can choose the entries of $S_{[l]}$ in such a way that $c_{-v+l}(\Theta_{[l+1]})_{ij}=0$ for all $(i,j)\in \cI_l$.

By the formula \eqref{eq:theta_l+1} above, we have
\begin{align*}
c_{-v+l}(\Theta_{[l+1]}) &= c_{-v+l}(\Theta_{[l]}) - S_{[l]}\cdot c_{-v}(\Theta_{[l]})^{(-l)} + c_{-v}(\Theta_{[l]})\cdot S_{[l]}^{(v)},
\end{align*}
and hence, taking into account that $ c_{-v}(\Theta_{[l]})$ is lower triangular, we obtain for all $1\leq i<j\leq d$,
\begin{align*}
c_{-v+l}(\Theta_{[l+1]})_{ij} &= c_{-v+l}(\Theta_{[l]})_{ij} - \sum_{k=j}^d (S_{[l]})_{ik}\cdot c_{-v}(\Theta_{[l]})^{(-l)}_{kj} + \sum_{k=1}^i c_{-v}(\Theta_{[l]})_{ik}\cdot (S_{[l]})^{(v)}_{kj}\\
&= -(S_{[l]})_{ij}\cdot c_{-v}(\Theta_{[l]})^{(-l)}_{jj} + c_{-v}(\Theta_{[l]})_{ii}\cdot (S_{[l]})^{(v)}_{ij} + c_{-v+l}(\Theta_{[l]})_{ij} \\
& \qquad - \sum_{k=j+1}^d (S_{[l]})_{ik}\cdot c_{-v}(\Theta_{[l]})^{(-l)}_{kj} + \sum_{k=1}^{i-1} c_{-v}(\Theta_{[l]})_{ik}\cdot (S_{[l]})^{(v)}_{kj}.
\end{align*}
So the condition $c_{-v+l}(\Theta_{[l+1]})_{ij}=0$ is a polynomial condition on $(S_{[l]})_{ij}$ of degree $q^v$ (as $c_{-v}(\Theta_{[l]})_{ii}\ne 0$) depending on entries $(S_{[l]})_{mn}$ with $n-m>j-i$.

We can, therefore, successively solve these equations for all $(i,j)\in \cI_l$ starting with $(i,j)=(1,d)$.
\end{proof}

\begin{rem}
    The proof also works for $v<0$ (with minor adaptations), but not for $v=0$.
\end{rem}

\begin{cor}\label{cor:orders-of-H}
A matrix $H$ as given in Theorem \ref{thm:conjugation} satisfies
\[ \ordsigma(H)=-v = \ordsigma(H_{ii}) \text{ for all $i=1,\ldots, d$}. \]
\end{cor}

\begin{proof}
 Since $\U(\vect{k})\subseteq \GL_d(\bK\sp{\sigma})$, the matrix $G$ satisfies $\ordsigma(G)=0$ and $c_0(G)\in \GL_d(\bK)$.
Hence,
\[ \ordsigma(H)=\ordsigma(G^{-1}\Theta G)\geq \ordsigma(G^{-1})+\ordsigma(\Theta)+\ordsigma(G)=-v,\]
and
\[ c_{-v}(H)=c_0(G^{-1})\cdot c_{-v}(\Theta)\cdot c_0(G^{-1})\in \GL_d(\bK).\]
In particular, $c_{-v}(H)\ne 0$ and therefore, $\ordsigma(H)=-v$.
The condition $\ordsigma(H_{ij})>-v+k_j-k_i\geq -v$ for all $(i,j)$ with $1\leq i<j\leq d$ ensures that $c_{-v}(H)$ is lower triangular.
Since $c_{-v}(H)\in \GL_d(\bK)$, this implies that the diagonal entries of $c_{-v}(H)$ are non-zero, and hence $\ordsigma(H_{ii})=-v$ for all $i=1,\ldots, d$.
\end{proof}

\section{Anderson \texorpdfstring{$t$}{t}-modules}\label{sec:anderson-t-modules}

    In this section, we recall the essential definition and terminologies of Anderson's objects. We begin with the definition of $t$-motives. 
    Let $\mot$ be a left $K\{\tau,t\}$-module. Then we say that $\mot$ is an abelian \emph{$t$-motive} over $K$ if $\mot$ is free of finite rank over both of $K[t]$ and $K\sp{\tau}$, and the quotient module $\mot/(K\{\tau\}\tau\mot)$ is a $(t-\theta)$-power-torsion module. Morphisms between abelian $t$-motives are left $K\{\tau,t\}$-module homomorphisms. 
    It was proved by Anderson \cite[Lemma~1.4.5]{ga:tm} that if a left $K\sp{\tau,t}$-module $\mot$ is finitely generated over both of $K[t]$ and $K\sp{\tau}$, then $\mot$ is free over $K[t]$ if and only if it is free over $K\sp{\tau}$,
    provided that $K$ is perfect. The category of abelian $t$-motives plays a crucial role in the present article.
    
    For $d>0$, a $d$-dimensional \emph{Anderson $t$-module} is an $\Fq$-algebra homomorphism
    \begin{align*}
        \phi:\Fq[t]&\to\Mat_d(K\{\tau\})\\
        a&\mapsto\phi_a
    \end{align*}
    such that $\partial\phi_t-\theta\mathbb{I}_d$ is a nilpotent matrix where $\partial(\sum_{i=0}^NB_i\tau^i):=B_0$ for $B_i\in\Mat_d(K)$. For each $d$-dimensional Anderson $t$-module $\phi$, we can associate a left $K\{\tau,t\}$-module $\mot_\phi:=\Mat_{1\times d}(K\{\tau\})$ with the natural multiplication by $K\{\tau\}$ from left, and the structure of $K[t]$-module is uniquely determined by
    \[
        t\cdot(x_1,\dots,x_d):=(x_1,\dots,x_d)\phi_t.
    \]
    Since $\partial\phi_t-\theta\mathbb{I}_d$ is a nilpotent matrix, we have $(t-\theta)^d(\mot_\phi/(K\{\tau\}\tau\mot_\phi))=0$. If $\mot_\phi$ defines an abelian $t$-motive, then we call the $t$-module $\phi$ abelian.
    In this case, we denote by $r:=\mathrm{rank}_{K[t]}\mot_\phi$ the rank of the abelian $t$-module $\phi$.

    Given two abelian $t$-modules $\phi$ and $\psi$ of dimension $d$ and $e$, a morphism $U:\phi\to\psi$ of $t$-modules is a matrix $U\in\Mat_{e\times d}(K\sp{\tau})$ such that $U\phi_a=\psi_aU$ for any $a\in\Fq[t]$. It induces a morphism between their $t$-motives given by the left $K\{\tau,t\}$-module homomorphism
    \begin{align*}
        U^{\dag}:\mot_\psi&\to\mot_\phi\\
        (x_1,\dots,x_e)&\mapsto(x_1,\dots,x_e)U.
    \end{align*}
    In this way, it was proved by Anderson \cite[Theorem~1]{ga:tm} that the functor $\phi\mapsto\mot_\phi$ gives the anti-equivalence between the category of abelian $t$-modules and the category of abelian $t$-motives.

    The inverse of the above functor can be described as follows. Let $\mot$ be an abelian $t$-motive with a fixed $K\{\tau\}$-basis $\{\kappa_1,\dots,\kappa_d\}\subset\mot$. Then there is a matrix $D_\kappa\in\Mat_d(K\{\tau\})$ such that
    \[
        t\begin{pmatrix}
            \kappa_1\\
            \vdots\\
            \kappa_d
        \end{pmatrix}=D_\kappa\begin{pmatrix}
            \kappa_1\\
            \vdots\\
            \kappa_d
        \end{pmatrix}.
    \]
    If we consider the $\mathbb{F}_q$-algebra homomorphism $\phi:\mathbb{F}_q[t]\to\Mat_d(K\{\tau\})$ uniquely determined by $\phi_t:=D_\kappa$, the matrix $\partial\phi_t-\theta\mathbb{I}_d$ must be a nilpotent matrix, since the quotient module $\mot/(K\{\tau\}\tau\mot)$ is a $(t-\theta)$-power-torsion module. Consequently, $\phi$ is indeed a $d$-dimensional $t$-module. 

    A different choice of $K\sp{\tau}$-basis $\{\varepsilon_1,\dots,\varepsilon_d\}\subset \mot$ induces another $t$-module $\psi:\mathbb{F}_q[t]\to\Mat_{d}(K\{\tau\})$ by setting $\psi_t=D_{\varepsilon}\in\Mat_d(K\{\tau\})$ where $D_\varepsilon$ is obtained in the same way as $D_\kappa$. 

    We can directly see that these $t$-modules are isomorphic. Indeed, note that there is a change of basis matrix $U\in\GL_d(K\{\tau\})$ such that
    \[
        \begin{pmatrix}
            \varepsilon_1\\
            \vdots\\
            \varepsilon_d
        \end{pmatrix}=U\begin{pmatrix}
            \kappa_1\\
            \vdots\\
            \kappa_d
        \end{pmatrix}.
    \]
    On the one hand,
    \[
        t\begin{pmatrix}
            \varepsilon_1\\
            \vdots\\
            \varepsilon_d
        \end{pmatrix}=\psi_t\begin{pmatrix}
            \varepsilon_1\\
            \vdots\\
            \varepsilon_d
        \end{pmatrix}=\psi_tU\begin{pmatrix}
            \kappa_1\\
            \vdots\\
            \kappa_d
        \end{pmatrix}.
    \]
    On the other hand,
    \[
        t\begin{pmatrix}
            \varepsilon_1\\
            \vdots\\
            \varepsilon_d
        \end{pmatrix}=tU\begin{pmatrix}
            \kappa_1\\
            \vdots\\
            \kappa_d
        \end{pmatrix}=Ut\begin{pmatrix}
            \kappa_1\\
            \vdots\\
            \kappa_d
        \end{pmatrix}=U\phi_t\begin{pmatrix}
            \kappa_1\\
            \vdots\\
            \kappa_d
        \end{pmatrix}.
    \]
    So, it follows that $\psi_tU=U\phi_t$, and thus $\phi$ and $\psi$ are isomorphic $t$-modules. We summarize the above observation as follows.

    \begin{lem}\label{lem:change-of-basis-is-isomorphic}
        Let $\phi$ and $\psi$ be two $d$-dimensional $t$-modules. If $\phi$ and $\psi$ correspond to the same $t$-motive $\mot$ (but with respect to different choices of $K\{\tau\}$-bases),
        then the $t$-modules $\phi$ and $\psi$ are isomorphic.
    \end{lem}

    Following Anderson \cite[pp.~467-468]{ga:tm}, we call the abelian $t$-motive $\mot$ \emph{pure} if there exists a $\bK\ps{\frac{1}{t}}$-lattice $\tilde{\Lambda}$ in $\bK\ls{\frac{1}{t}}\otimes_{K[t]} \mot$ and positive integers $u,v$ such that
    \[
        t^u\tilde{\Lambda}=\tau^v\tilde{\Lambda}.
    \]
    Here, as before, $\bK$ denotes an algebraic closure of $K$.
    The \emph{weight} of a pure $t$-motive $\mot$ is defined by $w(\mot):=u/v$. A $d$-dimensional abelian $t$-module $\phi$ of rank $r$ is called pure if $\mot_\phi$ is a pure $t$-motive. The weight of $\phi$ is defined to be $w(\phi):=d/r$. It was proved by Anderson \cite[Lemma~1.10.1]{ga:tm} that for a pure $t$-module $\phi$, we have $w(\phi)=w(\mot_\phi)$.

    Observe that the purity and weight of $t$-motives and $t$-modules only depend on their isomorphism classes.

    In \cite{cn-mp:hpqpam}, Namoijam and Papanikolas defined the notion of strictly pure and almost strictly pure $t$-modules. Then they proved that they are always pure $t$-modules.\footnote{As mentioned in the introduction, for strictly pure $t$-modules this was already shown in \cite[Remark~5.5.5]{dg:bsffa} and \cite[Theorem~7.3.6]{dt:ffa}.}
    
    \begin{defn} (\cite[Example 3.38 \& Section 4.5]{cn-mp:hpqpam})\\ \label{defn:almost-strictly-pure-t-module}
        A $d$-dimensional $t$-module $\phi$ is called \emph{almost strictly pure} if there is a positive integer $n$ so that the leading matrix of $\phi_{t^n}$ is invertible, that is, when we express
    \[
        \phi_{t^n}=B_0+B_1\tau+\cdots+B_s\tau^s,
    \]
    with $B_s\ne 0$, then $B_s\in\GL_d(K)$. If the exponent $n$ can be chosen as $1$, then the $t$-module $\phi$ is further called \emph{strictly pure}. 
    \end{defn}
    
    By the anti-equivalence between abelian $t$-modules and abelian $t$-motives, we can transport the above notion to the $t$-motive side.

    \begin{defn}
        We call a $t$-motive $\mot$ \emph{almost strictly pure with respect to the $K\sp{\tau}$-basis $\{\kappa_1,\dots,\kappa_d\}$} if there is a positive integer $n$ such that
        \[
            t^n\begin{pmatrix}
                \kappa_1\\
                \vdots\\
                \kappa_d
            \end{pmatrix}=D\begin{pmatrix}
                \kappa_1\\
                \vdots\\
                \kappa_d
            \end{pmatrix}
        \]
        for some $D=D_0+D_1\tau+\cdots+D_s\tau^s$ with $D_s\in\GL_d(K)$. If the exponent $n$ can be chosen as $1$, then the $t$-motive $\mot$ is further called \emph{strictly pure with respect to the basis $\{\kappa_1,\dots,\kappa_d\}$}.
    \end{defn}

    As mentioned in the introduction, the strictly pure $t$-motives are almost strictly pure, and all almost strictly pure $t$-motives are pure (see \cite[Remark~4.87]{cn-mp:hpqpam}). Our main results Theorem~\ref{thm:pure-implies-almost-strictly-pure} and Corollary~\ref{cor:almost-strictly-pure-over-finite-extension} provide the criterion for the reverse direction when pure $t$-motives are almost strictly pure or even strictly pure.
    
    The following proposition is immediate from the definition.

    \begin{prop}
        Let $\phi$ be a $d$-dimensional $t$-module. Then $\phi$ is almost strictly pure (resp. strictly pure) if and only if $\mot_\phi$ is almost strictly pure (resp. strictly pure) with respect to the standard basis.
    \end{prop}

Given two pure $t$-motives $\mot_1$ and $\mot_2$, their \emph{tensor product} is defined to be the $K[t]$-module $\mot:=\mot_1\otimes_{K[t]}\mot_2$ with diagonal $\tau$-action, and $\mot$ is again a pure $t$-motive (see \cite[Section 1.11]{ga:tm}).
Moreover, the weight of the $\mot$ is given by $w(\mot)=w(\mot_1)+w(\mot_2)$ (see \cite[Prop.~1.11.1]{ga:tm}). 

    Accordingly, the \emph{tensor product} $\phi\otimes \psi$ of two pure $t$-modules $\phi$ and $\psi$ is defined up to isomorphism as a $t$-module that corresponds to the tensor product $\mot_\phi\otimes_{K[t]}\mot_\psi$ with a fixed choice of $K\sp{\tau}$-basis.
In particular, given two pure $t$-modules $\phi$ and $\psi$, its tensor product $\phi\otimes\psi$ is again a pure $t$-module.

\section{Main Theorems}\label{sec:main-theorem}

We are now prepared to prove the main theorems that we already cited in the introduction. We start with the result on pure $t$-motives over an algebraically closed field $\bK$.

Since the definition of purity uses $\bK\ps{t}$-lattices in $\bK\ls{\frac{1}{t}}\otimes_{K[t]} \mot$, but more suitable to our tasks are $\bK\sps{\sigma}$-lattice in $ \bK\sls{\sigma}\otimes_{K\sp{\tau}} \mot$, we need a result from \cite{am:aefam} connecting these two.

\begin{prop}\label{prop:iso-of-completions} (cf.~\cite[Proposition 7.8]{am:aefam})\\
Let $\mot$ be an abelian Anderson $t$-motive over $\bK$, and denote
\[ \mothat:=\bK\sls{\sigma}\otimes_{\bK\sp{\tau}} \mot
\quad\text{and}\quad
 \hat{\mot}^{1/t}:=\bK\ls{\tfrac{1}{t}}\otimes_{\bK[t]} \mot
.\] Then
\begin{enumerate}
    \item There is a natural isomorphism of $\bK\ls{\frac{1}{t}}\{\tau\}$-modules $\iota:\hat{\mot}^{1/t} \to \mothat$.
    \item \label{prop:iso-of-completions:item:2} If $\Lambda$ is a $\bK\sps{\sigma}$-lattice in $\mothat$ such that $t^{-1}\Lambda\subseteq \Lambda$, then $\tilde{\Lambda}:=\iota^{-1}(\Lambda)\subseteq \hat{\mot}^{1/t}$ is a $\bK\ps{\frac{1}{t}}$-lattice in $\hat{\mot}^{1/t}$.
\item \label{prop:iso-of-completions:item:3} If $\tilde{\Lambda}$ is a $\bK\ps{\frac{1}{t}}$-lattice in $\hat{\mot}^{1/t}$ such that $\sigma\tilde{\Lambda}\subseteq \tilde{\Lambda}$, then $\Lambda:=\iota(\tilde{\Lambda})\subseteq \mothat$ is a $\bK\sps{\sigma}$-lattice in $\mothat$.
\end{enumerate}
\end{prop}

From this proposition, we deduce the following.

\begin{cor}\label{cor:stable-lambda}
    Let $\mot$ be a pure Anderson $t$-motive over $\bK$ of $\bK\sp{\tau}$-rank $d$ and $\bK[t]$-rank $r$. Then there exist $u,v\geq 0$ and a $\bK\sps{\sigma}$-lattice $\Lambda$ in $\mothat=\bK\sls{\sigma}\otimes_{\bK\sp{\tau}} \mot$ such that
    \[  t^u \Lambda=\tau^v \Lambda\quad \text{ and }\quad \frac{u}{v}=\frac{d}{r}.\]
\end{cor}

Although the proof can be extracted from the last paragraph in the proof of \cite[Theorem 7.2]{am:aefam}, we give the proof here for the convenience of the reader. 

\begin{proof}  
Since $\mot$ is pure, there exist positive numbers $u,v\in \ZZ$, and a $\bK\ps{\frac{1}{t}}$-lattice $\Lambda'$ in $\bK\ls{\frac{1}{t}}\otimes_{K[t]} \mot$ such that
\[ t^u\Lambda' =\tau^v \Lambda', \]
and by \cite[Lemma 1.10.1]{ga:tm}, we have $\frac{u}{v}=\frac{d}{r}$.

Then $\tilde{\Lambda}:=\sum\limits_{i=0}^{v-1} \sigma^i\Lambda'$ is another $\bK\ps{\frac{1}{t}}$-lattice that satisfies $t^u\tilde{\Lambda}=\tau^v\tilde{\Lambda}$. It additionally satisfies
$\sigma\tilde{\Lambda}\subseteq \tilde{\Lambda}$, since
\[ \sigma^v\Lambda'=t^{-u}\Lambda'\subseteq \Lambda'\subseteq \tilde{\Lambda}.\]
Hence by Proposition \ref{prop:iso-of-completions}\eqref{prop:iso-of-completions:item:3}, $\Lambda:=\iota(\tilde{\Lambda})\subseteq \mothat$ is a $\bK\sps{\sigma}$-lattice in $\mothat$. Since the isomorphism $\iota:K\ls{\frac{1}{t}}\otimes_{K[t]} \mot \to \mothat$ is compatible with the $t$-action and the $\tau$-action, this lattice $\Lambda$ satisfies $t^u\Lambda=\tau^v\Lambda$.
\end{proof}

\begin{thm}\label{thm:pure-implies-almost-strictly-pure}
Let $\mot$ be a pure Anderson $t$-motive over $\bK$ of weight $w\in\mathbb{Q}$.
Then there exists a $\bK\sp{\tau}$-basis $\basis{\tilde{\kappa}}{d}$ of $\mot$, so that $\mot$ is almost strictly pure with respect to that basis. In addition, the $\bK\{\tau\}$-basis can be chosen to which $\mot$ is strictly pure, if and only if the reciprocal of the weight of $\mot$ is an integer.
\end{thm}

\begin{proof}
Let $\mot$ be a pure Anderson $t$-motive over $\bK$ of $\bK\sp{\tau}$-rank $d$ and $\bK[t]$-rank $r$.
By Corollary \ref{cor:stable-lambda}, there exist positive numbers $u,v\in \ZZ$, and a $\bK\sps{\sigma}$-lattice $\Lambda$ in $\mothat= \bK\sls{\sigma}\otimes_{K\sp{\tau}} \mot$ satisfying $t^u\Lambda=\tau^v\Lambda$, and $\frac{u}{v}=\frac{d}{r}$. In addition, if $\frac{1}{w}\in\mathbb{Z}$, then by \cite[Proposition~2.3.11]{uh-akj:pthshcff}, we can choose $\Lambda$ with $u=1$ and $v=w$.

Let $\vect{b}\in\Mat_{d\times 1}(\Lambda)$ be any $\bK\sps{\sigma}$-basis of $\Lambda$ and $\vect{\kappa}\in\Mat_{d\times 1}(\mot)$ be any $\bK\sp{\tau}$-basis of $\mot$.
As $\vect{b}$ and $\vect{\kappa}$ are both $\bK\sls{\sigma}$-bases of $\mothat$, there is a matrix $C\in \GL_d(\bK\sls{\sigma})$ such that
\[  \vect{b}=C\cdot \vect{\kappa}. \]
By Theorem \ref{thm:matrix-decomposition}, there exist $A\in \GL_d(\bK\sps{\sigma})$, $B\in \GL_d(\bK\sp{\tau})$, and a diagonal matrix $D$ with $\sigma$-powers on the diagonal such that
    \[ C= A\cdot D \cdot B ,\] 
and we can even achieve that
\[ D = \Diagonal{\sigma^{-k_1}}{\sigma^{-k_d}} = \Diagonal{\tau^{k_1}}{\tau^{k_d}} \]
with $k_1\leq k_2\leq \dots \leq k_d$.
    
Then $\vect{b}':=A^{-1}\vect{b}$ is another $\bK\sps{\sigma}$-basis of the lattice $\Lambda$, and $\vect{\kappa}':=B\vect{\kappa}$ is another $\bK\sp{\tau}$-basis of the $t$-motive $\mot$, and we have $\vect{b}'=D\cdot \vect{\kappa}'$.

\medskip

Let $\Theta\in \GL_d(\bK\sls{\sigma})$ be such that $t^u\cdot \vect{b}'=\Theta \vect{b}'$.
Since $\vect{b}'$ is a basis of $\Lambda$, and $\sigma^{v}t^u\Lambda=\Lambda$, we have $\sigma^{v}\Theta\in \GL_d(\bK\sps{\sigma})$. Therefore,
$\ordsigma(\Theta)=-v$ and $c_{-v}(\Theta)\in \GL_d(\bK)$.

By Theorem \ref{thm:conjugation}, there is a matrix $G\in \U(\vect{k})$, where $\vect{k}=(k_1,\ldots, k_d)$ such that the matrix
$H:=G\Theta G^{-1}$ satisfies the condition
\[  \ordsigma(H_{ij})>-v+k_j-k_i \quad \forall 1\leq i <j\leq d. \]
By Corollary \ref{cor:orders-of-H}, $H$ also satisfies
\[ \ordsigma(H)=-v = \ordsigma(H_{ii}) \text{ for all $i=1,\ldots, d$}. \]

Now, let 
\[ \tilde{G}:=D^{-1}GD\quad \text{and} \quad \tilde{H}:=D^{-1}HD. \]

For the entries of $\tilde{G}$, we obtain $\tilde{G}_{ij}=\sigma^{k_i}G_{ij}\tau^{k_j}$ for all $1\leq i,j\leq d$. 
Since $G$ is unipotent upper triangular, this also holds for $\tilde{G}$. Further, since $G\in \U(\vect{k})$, for all $1\leq i<j\leq d$,
\[   \degsigma(\tilde{G}_{ij}) = k_i+\degsigma(G_{ij})-k_j\leq k_i+(k_j-k_i)-k_j=0. \]
Hence, the entries of $\tilde{G}$ above the diagonal are in $\bK\sp{\tau}$. In particular, $\tilde{G}\in \GL_d(\bK\sp{\tau})$.

For $\tilde{H}$, we similarly have $\tilde{H}_{ij}=\sigma^{k_i}H_{ij}\tau^{k_j}$ for all $1\leq i,j\leq d$, and therefore,

\begin{align*}
    \ordsigma(\tilde{H}_{ij}) &= 
    \partdef{ k_i+\ordsigma(H_{ij})-k_j >  k_i+(-v+k_j-k_i)-k_j= -v & \text{for }i<j, \\
    \ordsigma(H_{ii}) = -v & \text{for }i=j, \\
    k_i+\ordsigma(H_{ij})-k_j \geq k_i-v-k_j\geq -v & \text{for }i>j.    
    }
\end{align*} 
Hence,
\begin{equation}\label{eq:order-of-tildeH}
    \ordsigma(\tilde{H})=-v\,\,\text{ and }\,\, c_{-v}(\tilde{H})\in \GL_d(\bK)
\end{equation}  
($c_{-v}(\tilde{H})$ is even lower triangular).

Since $\tilde{G}\in \GL_d(\bK\sp{\tau})$, $\tilde{\vect{\kappa}}:=\tilde{G}\vect{\kappa}'$ is another $\bK\sp{\tau}$-basis of $\mot$, and we have
\begin{align*}
t^u \tilde{\vect{\kappa}} &= t^u \tilde{G}D^{-1}\vect{b}'  = \tilde{G}D^{-1} \cdot t^u \vect{b}' \\
 &= \tilde{G}D^{-1} \Theta \vect{b}' = \tilde{G}D^{-1} \Theta D \tilde{G}^{-1} \tilde{\vect{\kappa}}\\
 &=  D^{-1}GD D^{-1} \Theta D D^{-1}G^{-1}D\tilde{\vect{\kappa}} = D^{-1} H D \tilde{\vect{\kappa}}\\
 &= \tilde{H}\tilde{\vect{\kappa}}
\end{align*}

This implies that in fact $\tilde{H}\in \Mat_d(\bK\sp{\tau})$ (since $t^u\tilde{\vect{\kappa}}$ consists of elements of $\mot$), and 
the conditions \eqref{eq:order-of-tildeH} mean that $\degtau(\tilde{H})=v$, and that its top coefficient matrix is invertible.

Hence, $\mot$ is almost strictly pure with respect to the basis $\tilde{\vect{\kappa}}$. For the case of $\frac{1}{w}\in\mathbb{Z}$, since we can choose $\Lambda$ with $u=1$ and $v=w$, $\mot$ is strictly pure with respect to the basis $\tilde{\vect{\kappa}}$. Finally, by \cite[Remark~4.87]{cn-mp:hpqpam}, if $\mot$ is strictly pure, then the weight of $\mot$ is of the form $\frac{1}{\ell}$ for some $\ell\in\mathbb{Z}$. In other word, the reciprocal of the weight of $\mot$ is an integer. This completes the proof.
\end{proof}

From Theorem \ref{thm:pure-implies-almost-strictly-pure}, we easily get a result for arbitrary fields $K$ containing $\Fq$.

\begin{cor}\label{cor:almost-strictly-pure-over-finite-extension}
 Let $\mot$ be a pure Anderson $t$-motive over $K$ of $K\sp{\tau}$-rank $d$ and weight $w\in\mathbb{Q}$.
Then there exists a finite field extension $L$ of $K$, and an 
$L\sp{\tau}$-basis $\basis{\tilde{\kappa}}{d}$ of $\mot_L:=L\sp{\tau}\otimes_{K\sp{\tau}}\mot$, so that $\mot_L$ is almost strictly pure with respect to $\basis{\tilde{\kappa}}{d}$. Moreover, the $L\{\tau\}$-basis can be chosen to which $\mot$ is strictly pure if and only if $\frac{1}{w}\in\mathbb{Z}$.
\end{cor}

\begin{proof}
Let $\basis{\kappa}{d}$ be any $K\sp{\tau}$-basis of $\mot$.

  Let $\bK$ be an algebraic closure of $K$ and consider the base extension $\mot_{\bK}:=\bK\otimes_K \mot$ with diagonal $\tau$-action, and with $t$-action induced from that on $\mot$.
  Of course, $\mot_{\bK}$ is then a pure Anderson $t$-motive over $\bK$, and $\basis{\kappa}{d}$ is a $\bK\sp{\tau}$-basis of $\mot_{\bK}$.

  By Theorem \ref{thm:pure-implies-almost-strictly-pure}, there is another $\bK\sp{\tau}$-basis $\basis{\tilde{\kappa}}{d}$ of $\mot_{\bK}$ with respect to which $\mot_{\bK}$ is almost strictly pure. The $\bK\{\tau\}$-basis can be chosen to which $\mot_{\bK}$ is strictly pure if and only if the reciprocal of the weight of $\mot_{\bK}$ is an integer.
  Let $G\in \GL_d(\bK\sp{\tau})$ be the base change matrix, i.e.,
  \[   \svect{\tilde{{\kappa}}}{d}= G\svect{\kappa}{d}. \]
Since each entry of $G$ is a polynomial in $\tau$ which has only finitely many coefficients (and the same holds for its inverse $G^{-1}\in \GL_d(\bK\sp{\tau})$), there exists a finite extension $L$ of $K$ such that $G\in \GL_d(L\sp{\tau})$.

This means that $\basis{\tilde{\kappa}}{d}$ is also a $L\sp{\tau}$-basis of the base extension $\mot_L=L\otimes_K \mot$. Therefore, $\mot_L$ is almost strictly pure with respect to $\basis{\tilde{\kappa}}{d}$, and is strictly pure if and only if $\frac{1}{w}\in\mathbb{Z}$. 
\end{proof}

Next, we translate this result on $t$-motives to a result on Anderson $t$-modules.

\begin{thm}\label{thm:almost-strictly-pure-over-extension}
    Let $K$ be a field containing $\Fq$, and let $\phi$ be a pure Anderson $t$-module over $K$ of dimension $d$ and weight $w$.
Then there exists a finite extension $L$ of $K$, and an almost strictly pure $t$-module $\psi$ over $L$ such that
the base extension of $\phi$ to $L$, i.e.~$\phi^L:\Fq[t]\xrightarrow{\phi} \Mat_d(K\sp{\tau})\hookrightarrow \Mat_d(L\sp{\tau})$ is isomorphic to $\psi$. The $t$-module $\psi$ over $L$ can be chosen to be strictly pure, if and only if $\frac{1}{w}\in\mathbb{Z}$.
\end{thm}

\begin{proof}
  Let $\mot:=\mot_\phi$ be the Anderson $t$-motive associated to $\phi$. So with respect to the standard basis $\basis{\kappa}{d}$, we have
  \[ t \cdot \svect{\kappa}{d} = \phi_t \cdot  \svect{\kappa}{d}.\]
  
  By Corollary \ref{cor:almost-strictly-pure-over-finite-extension}, there is a finite extension $L$ of $K$, and an $L\sp{\tau}$-basis $\basis{\tilde{\kappa}}{d}$ of $\mot_L$ such that $\mot_L$ is almost strictly pure with respect to that basis. The $L\{\tau\}$-basis can be chosen to which $\mot_L$ is strictly pure if and only if $\frac{1}{w}\in\mathbb{Z}$.

 Now, choose $\psi$ to be the Anderson $t$-module over $L$ corresponding to $\mot_L$ with respect to the basis $\basis{\tilde{\kappa}}{d}$. This Anderson $t$-module is almost strictly pure, and is strictly pure if and only if $\frac{1}{w}\in\mathbb{Z}$.

 Since $\phi^L$ and $\psi$ are both $t$-modules corresponding to $\mot_L$, they are isomorphic by Lemma~\ref{lem:change-of-basis-is-isomorphic}.
\end{proof}

\begin{rem}\label{rem:necessity-of-field-extension}
    At the writing of this paper, it is not clear to the authors when a field extension is really needed. Our algorithm provides a procedure to obtain an almost strictly pure form, but one could also obtain such a form differently. This is also illustrated in Example \ref{ex:weight-1/4}.
    
    We would like to find criteria about this issue in future work.
\end{rem}

\section{Applications}\label{sec:applications}
    In this section, we aim to present two applications of our main results. The first one concerns the tensor product of almost strictly pure $t$-modules, and the second one extends Poonen's Mordell-Weil theorem to pure $t$-modules.

    \subsection{Tensor products of $t$-modules}
    Let $\phi$ and $\psi$ be two almost strictly pure $t$-modules defined over $\bK$. By \cite{cn-mp:hpqpam}, $\mot_\phi$ and $\mot_\psi$ are pure $t$-motives. Thus, $\mot_\phi\otimes_{\bK[t]}\mot_\psi$ defines a pure $t$-motive. When $\phi$ and $\psi$ are Drinfeld modules, there are explicit choices of $\bK\{\tau\}$-basis for $\mot_\phi\otimes_{\bK[t]}\mot_\psi$ (see \cite[\S 2]{yh:tpdmvr} and \cite[Theorem 4.7]{ck:ratpdmttp}) giving the almost strict pureness, while it is relied on some ad. hoc. calculations (see \cite[Remark 4.2.12]{wch:cgltpdm}).\footnote{Be aware that there are some small typos in the computations, e.g.~the $i$-th diagonal entry of the specified matrix $\tilde{B}_r$ is $\kappa_r\kappa_r^{(i-1)}$ and not $\kappa_r^2$.}

     The first application of our main result provides a uniform way to verify the almost strict pureness for $\mot_\phi\otimes_{\bK[t]}\mot_\psi$.

    \begin{thm}\label{thm:tensor-products}
        Let $\phi$ and $\psi$ be two almost strictly pure $t$-modules defined over $\bK$ and $\mot:=\mot_\phi\otimes_{\bK[t]}\mot_\psi$ be the tensor product of their $t$-motives. Then there is an almost strictly pure $t$-module $\rho$ defined over $\bK$ such that
        \[
            \mot\cong\mot_\rho.
        \]
        In other words, $\mot_\phi\otimes_{\bK[t]}\mot_\psi$ is almost strictly pure with respect to an appropriate $\bK\{\tau\}$-basis.
    \end{thm}

    \begin{proof}
        Since $\mot=\mot_\phi\otimes_{\bK[t]}\mot_\psi$ defines a pure $t$-motive, by Theorem~\ref{thm:pure-implies-almost-strictly-pure} there is a $\bK\{\tau\}$-basis $(\kappa_1,\dots,\kappa_d)^\tr\in\Mat_{d\times 1}(\mot)$ for $\mot$ and a positive integer $n$ such that
        \[
            t^n\begin{pmatrix}
                \kappa_1\\
                \vdots\\
                \kappa_d
            \end{pmatrix}=D\begin{pmatrix}
                \kappa_1\\
                \vdots\\
                \kappa_d
            \end{pmatrix}
        \]
        for some $D=D_0+D_1\tau+\cdots+D_s\tau^s\in\Mat_d(\bK\{\tau\})$ with $D_s\in\GL_d(\bK)$. Let $B\in\Mat_d(\bK\{\tau\})$ be the matrix representing the $t$-action on $\mot$, that is, $t(\kappa_1,\dots,\kappa_d)^\tr=B(\kappa_1,\dots,\kappa_d)^\tr$. If we define the $t$-module $\rho$ by setting $\rho_t=B$, then $\rho$ is almost strictly pure since $\rho_{t^n}=D$. Furthermore, we have $\mot\cong\mot_\rho$ by construction. The desired result now follows.        
    \end{proof}

    \begin{exmp}\label{exmp:tensor-product-of-rank-2-DFs}
    Let $\phi$ and $\psi$ be two Drinfeld modules over $\bK$ of rank $2$, i.e,~
    \[  \phi_t=\theta+\alpha_1\tau+\alpha_2\tau^2\quad\text{and}\quad\psi_t=\theta+\beta_1\tau+\beta_2\tau^2\]
    with $\alpha_1,\alpha_2,\beta_1,\beta_2\in \bK$, $\alpha_2,\beta_2\ne 0$.

If we take the tensor product $\rho$ of these two Drinfeld modules of rank $2$. It has rank $2\cdot 2=4$ and dimension $2+2=4$, so its weight is $\frac{4}{4}=1$. Using the computation in Khaochim's PhD thesis \cite{ck:ratpdmttp}, one obtains
\[ \rho_t =\begin{pmatrix}
    \theta & 0 & \alpha_1 & \alpha_2 \\
    \alpha_1\tau & \theta & \alpha_2\tau & 0\\
    \beta_1\tau & \beta_2\tau & \theta & 0 \\
    \beta_2\tau^2 & 0 & \beta_1\tau & \theta
\end{pmatrix},\]
and readily computes that $\rho_t^2$ has as leading term the matrix
\[ \begin{pmatrix}
    \alpha_2\beta_2 & 0 & 0 & 0 \\
    \alpha_2\beta_1^{(1)} & \alpha_2\beta_2^{(1)} & 0 & 0 \\
    \beta_2\alpha_1^{(1)} & 0 & \beta_2\alpha_2^{(1)} & 0 \\
     \beta_2\theta^{(2)}+\beta_1\beta_1^{(1)} & \beta_1\beta_2^{(1)} & \beta_2\alpha_1^{(2)} & \beta_2\alpha_2^{(2)}
\end{pmatrix} \tau^2.\]
This verifies that $\rho$ is not strictly pure, but almost strictly pure.

However, Theorem \ref{thm:pure-implies-almost-strictly-pure} implies that $\rho$ is isomorphic to a strictly pure $t$-module.

To compute such a $t$-module, first take the standard $\bK\sp{\tau}$-basis $\vect{\kappa}$ for the $t$-motive such that $t\cdot \vect{\kappa}=\rho_t \vect{\kappa}$, and let $\vect{b}$ such that $b_4=\tau \kappa_1$, and $b_3=\kappa_2$, $b_2=\kappa_3$, $b_1=\kappa_4$.
Then
\[  t\cdot \vect{b} = \begin{pmatrix}
    \theta & \beta_1\tau & 0 & \beta_2\tau \\
    0 &\theta & \beta_2\tau & \beta_1 \\
    0 & \alpha_2\tau & \theta & \alpha_1 \\
    \alpha_2^q \tau & \alpha_1^q\tau & 0 & \theta^q
\end{pmatrix} \cdot \vect{b},\]
and we see that the coefficient matrix of $\tau$ is invertible. Hence, the $\bK\sps{\sigma}$-lattice $\Lambda$ spanned by $\vect{b}$ satisfies $t\Lambda=\tau\Lambda$.

After having found $\vect{b}$, we could apply the algorithm to compute a desired $\bK\sp{\tau}$-basis $\basis{\tilde{\kappa}}{4}$ for $\mot_\rho$.

However, in this explicit simple case, the general strategy reduces to killing the $\tau^2$-term in $\rho_t$ by conjugating with a matrix of the shape
\[  \tilde{G}=\begin{pmatrix}
    1 & && \\
    0 & 1 & & \\
    0 & 0 & 1 & \\
    -g\tau & 0 & 0 & 1
\end{pmatrix}\]
for an appropriate element $g\in \bK$.
Since,
\begin{align*}
    \tilde{G}\rho_t \tilde{G}^{-1} &= \begin{pmatrix}
    1 & && \\
    0 & 1 & & \\
    0 & 0 & 1 & \\
    -g\tau & 0 & 0 & 1
\end{pmatrix}\cdot \begin{pmatrix}
    \theta & 0 & \alpha_1 & \alpha_2 \\
    \alpha_1\tau & \theta & \alpha_2\tau & 0\\
    \beta_1\tau & \beta_2\tau & \theta & 0 \\
    \beta_2\tau^2 & 0 & \beta_1\tau & \theta
\end{pmatrix}\cdot \begin{pmatrix}
    1 & && \\
    0 & 1 & & \\
    0 & 0 & 1 & \\
    g\tau & 0 & 0 & 1
\end{pmatrix} \\
&= 
\begin{pmatrix}
    \theta+\alpha_2 g\tau & 0 & \alpha_1  & \alpha_2 \\
    \alpha_1\tau & \theta & \alpha_2\tau & 0 \\
    \beta_1\tau & \beta_2\tau & \theta & 0 \\
    (\theta-\theta^q)g\tau & 0 & (\beta_1-g\alpha_1^q)\tau & \theta-g\alpha_2^q\tau
\end{pmatrix} +\begin{pmatrix}
    && &  \\
    & &&\\
    & & &\\
   \beta_2-g^{q+1}\alpha_2^q & & &  
\end{pmatrix}\tau^2,
\end{align*}
This means that $g$ has to be a $(q+1)$th root of $\frac{\beta_2}{\alpha_2^q}$, and in that case the $\tau$-coefficient matrix is indeed invertible.
\end{exmp}

    \begin{rem}
        As mentioned above, and as it is seen explicitly in Example \ref{exmp:tensor-product-of-rank-2-DFs}, in the case of two Drinfeld modules of rank $2$, the tensor product is almost strictly pure without the need of a field extension.
    
        An interesting question is whether this holds for the tensor product of any almost strictly pure $t$-modules which leads us to the following conjecture.      
    \end{rem}

\begin{conj}
    Given two almost strictly pure $t$-modules $\phi$ and $\psi$ over $K$. The tensor product $\phi\otimes \psi$ is almost strictly pure if we choose the $K\sp{\tau}$-basis of the corresponding $t$-motive appropriately.
\end{conj}

\begin{rem}
    It is worth to mention that Theorem~\ref{thm:tensor-products} can be also applied to the tensor product of finitely many almost strictly pure $t$-modules, by using some natural induction arguments. A good example of this result is the tensor powers of a fixed Drinfeld module.
\end{rem}

\subsection{Poonen's Mordell-Weil theorem for pure $t$-modules}
    The second application of our main result provides a structure theorem for the $\Fq[t]$-module of $K$-valued points on the pure $t$-module $\phi$ defined over $K$. Recall that Denis defined the global height functions on almost strictly pure $t$-modules in \cite{JTNB_1995__7_1_97_0}. Using the properties of Denis' height functions, Kuan \cite{Kuan2022-kh} proved that the $\Fq[t]$-module $\phi(K)$ is tame for any almost strictly pure $t$-module $\phi$ defined over $K$. This result extends the work of Poonen \cite{bp:lhfmwtdm} when $\phi$ is a Drinfeld modules. Here we say a module over a Dedekind domain $R$ tame if every $R$-submodule of finite rank is finitely generated as an $R$-module. As a consequence, Kuan's result extends Poonen's Mordell-Weil theorem for Drinfeld modules to any almost strictly pure $t$-module. By our main result established in this article, we generalized this structure theorem further to all pure $t$-modules.

    \begin{thm}\label{thm:MordellWeil}
        Let $\phi$ be a pure $t$-module defined over $K$. Then there is a finite extension $L$ over $K$ such that the base extension of $\phi$ to $L$ satisfies Poonen's Mordell-Weil theorem, that is, the $\Fq[t]$-module $\phi^L(L)$ is the direct sum of its finite torsion submodule with a free $\Fq[t]$-module of rank $\aleph_0$.
    \end{thm}

    \begin{proof}
        By Theorem~\ref{thm:almost-strictly-pure-over-extension}, there is a finite extension $L$ and an almost strictly pure $t$-module $\psi$ such that the base extension of $\phi$ to $L$ is isomorphic $\phi^L$, namely, there is a $L$-isomorphism $U\in\GL_d(L\{\tau\})$ such that $U\phi_t^L=\psi_tU$. By the result of Kuan, we have the $\mathbb{F}_q[t]$-module isomorphism
        \[
            \psi(L)\cong\mathbb{F}_q[t]^{\aleph_0}\oplus\psi(L)_{\mathrm{tor}}
        \]
        where $\psi(L)_{\mathrm{tor}}$ refers to the set of torsion elements of $\psi$ in $L$, which is a finite set. It follows that
        \[
            \phi^L(L)\cong U\big(\phi^L(L)\big)=\psi(L).
        \]
        is the direct sum of its finite torsion submodule with a free $\mathbb{F}_q[t]$-module of rank $\aleph_0$.
    \end{proof}

\section{Examples} \label{sec:examples}

\begin{exmp}\label{ex:weight-1/4}
Let us consider the Anderson $t$-module $\phi$ over $K=\Fq(\theta)$ given by
\[  \phi_t= \begin{pmatrix} \theta+\tau^4+\tau^5 & \tau^6 \\ \tau-\tau^4 & \theta+\tau^4-\tau^5 \end{pmatrix}.\]
If we let $\{\kappa_1,\kappa_2\}$ be the standard $K\sp{\tau}$-basis of the corresponding $t$-motive $\mot$, and $\{b_1,b_2\}$ be given by
\[  \vect{b} = \begin{pmatrix} 1 & 0 \\ \tau^2 & \tau^3 \end{pmatrix}\vect{\kappa}, \]
we compute
\begin{align*}
    t \vect{b} &= \begin{pmatrix} 1 & 0 \\ \tau^2 & \tau^3 \end{pmatrix}\cdot \phi_t\cdot \begin{pmatrix} 1 & 0 \\ \tau^2 & \tau^3 \end{pmatrix}^{-1} \cdot \vect{b} \\
&=    \begin{pmatrix} \theta+\tau^4 & \tau^3 \\ (\theta^{q^2}-\theta^{q^3})\tau^2+\tau^4 & \theta^{q^3}+\tau^4 \end{pmatrix} \vect{b}.
\end{align*} 
Hence, the $K\sps{\sigma}$-lattice $\Lambda$ in $\mothat$ spanned by $\vect{b}$ satisfies
\begin{equation}\label{eq:lattice-example}
     t\Lambda = \tau^4 \Lambda.
\end{equation}
Therefore, $\mot$ (and $\phi$) is pure of weight $\frac{1}{4}$, and by Theorem \ref{thm:pure-implies-almost-strictly-pure}, $\mot$ is even strictly pure with respect to an appropriate $K\sp{\tau}$-basis.

Following the algorithms in this article, we are going to find a basis $\vect{\kappa'}$ of $\mot$ with respect to which $\mot$ is strictly pure. 

First at all, we have to decompose $C:=\begin{pmatrix} 1 & 0 \\ \tau^2 & \tau^3 \end{pmatrix}$ into a product $A\cdot D\cdot B$ as in Theorem \ref{thm:matrix-decomposition}.

Following the algorithm in its proof, we obtain
\begin{align*}
     A&=\begin{pmatrix}
    -1 & \sigma^2 \\ 0 & 1
\end{pmatrix}^{-1} &=& \begin{pmatrix}
    -1 & \sigma^2 \\ 0 & 1
\end{pmatrix},\\
B &= \left( \begin{pmatrix}
    0 & 1 \\ 1 & 0
\end{pmatrix}\begin{pmatrix}
    1 & 0 \\ -\sigma^{-1} & 1
\end{pmatrix} \right)^{-1} &=& \begin{pmatrix}
    0 & 1 \\ 1 & \tau
\end{pmatrix},\\
D &= \begin{pmatrix}
    \sigma^{-1} & 0 \\ 0 & \sigma^{-2}
\end{pmatrix} &=& \begin{pmatrix}
    \tau & 0 \\ 0 & \tau^2
\end{pmatrix}.
\end{align*}
Further, the basis $\vect{b''}$ in the proof of Theorem \ref{thm:pure-implies-almost-strictly-pure}, is 
\[ \vect{b''}=A^{-1}\vect{b}=\begin{pmatrix}
    -1 & \sigma^2 \\ 0 & 1
\end{pmatrix}\cdot \begin{pmatrix} 1 & 0 \\ \tau^2 & \tau^3 \end{pmatrix}\vect{\kappa}
= \begin{pmatrix}
    0 & \tau \\ \tau^2 & \tau^3
\end{pmatrix} \vect{\kappa},
\]
and the matrix $\Theta\in \GL_d(\bK\sls{\sigma})$ satisfying
$t \vect{b''}=\Theta \vect{b''}$ is
\[ \Theta = \begin{pmatrix}
    \theta^q -\tau^2+\tau^4 & (\theta-\theta^q)\sigma^2+1- \tau^3 \\
    (\theta^{q^3}-\theta^{q^2})\tau^2 - \tau^4 & \theta^{q^2}+\tau^2+\tau^4
\end{pmatrix}\]
with $\ord_{\sigma}(\Theta)=-4$, and
\[ c_{-4}(\Theta)= \begin{pmatrix}
    1 & 0 \\ -1 & 1
\end{pmatrix}.\]

Next, we need to find the matrix $G\in \U(k_1,k_2)$ (where $k_1=1$, $k_2=2$ according to the entries of $D$) such that $H:=G\Theta G^{-1}$ satisfies
\[   \ordsigma(H_{ij})>-4+k_j-k_i \forall 1\leq i<j\leq 2,\]
i.e.~$\ordsigma(H_{12})>-4+2-1=-3$, following the algorithm in the proof of Theorem~\ref{thm:conjugation}.

As in our case, $d=2$ and $k_2-k_1=1$, the whole procedure simplifies to finding $S_{[l]}=\begin{pmatrix}
    0 & s_l \\ 0 & 0
\end{pmatrix}\in \Mat_d(\bK)$ for $l=0,1$ such that
first
\[ \Theta_{[1]}:=\begin{pmatrix}
    1 & -s_0 \\ 0 & 1
\end{pmatrix}\Theta \begin{pmatrix}
    1 & s_0 \\ 0 & 1
\end{pmatrix}\]
satisfies $\ordsigma((\Theta_{[1]})_{12})>-4$, and second
\[  \Theta_{[2]}:=\begin{pmatrix}
    1 & -s_1\sigma \\ 0 & 1
\end{pmatrix}\Theta \begin{pmatrix}
    1 & s_1\sigma \\ 0 & 1
\end{pmatrix}\]
satisfies $\ordsigma((\Theta_{[2]})_{12})>-3$.

If we closely follow the proof of Theorem \ref{thm:conjugation}, we would first look for $F\in \GL_2(\bK)$ such that $F^{-1}F^{(-4)}=c_{-4}(\Theta)=\begin{pmatrix}
    1 & 0 \\ -1 & 1
\end{pmatrix}$,
make an LU-decomposition of $F$ as $F=L\cdot U$ with $U$ unipotent, and then take $S_{[0]}=\one_2-U$.

However, since, we already have $\ordsigma(\Theta_{12})=-3>-4$, we can just take $s_0=0$ so that $\Theta_{[1]}=\Theta$.

The condition for $s_1$ is
\[ 0 = c_{-3}((\Theta_{[2]})_{12}) = s_1^{q^4}-s_1-1.\]
And hence, we take $s_1\in \Fq^{\textrm{alg}}$ solving this equation.\footnote{Apparently, there is no solution in $\Fq$ of that equation, as $x^q=x$ for all $x\in \Fq$.}
So the matrix $\tilde{G}$ is
\[  \tilde{G}=D^{-1}\begin{pmatrix}
    1 & -s_1\sigma \\ 0 & 1
\end{pmatrix}D = \begin{pmatrix}
    1 & -s_1^{1/q} \\ 0 & 1
\end{pmatrix},\]
and the desired basis $\vect{\kappa'}$ is given by
\begin{align*}
    \vect{\kappa'} &= \tilde{G}\vect{\kappa''}=\tilde{G}B\vect{\kappa} \\
    &= \begin{pmatrix}
        -s_1^{1/q} & 1-s_1^{1/q}\tau \\ 1 & \tau
    \end{pmatrix}\vect{\kappa}
\end{align*}
with $t$-action
$t\cdot \vect{\kappa'}=\tilde{H}\vect{\kappa'}$ where
\begin{align*}
    \tilde{H} &= \tilde{G}B \phi_t (\tilde{G}B)^{-1}\\
    &= \begin{pmatrix}
        -s_1^{1/q} & 1-s_1^{1/q}\tau \\ 1 & \tau
    \end{pmatrix}  \begin{pmatrix} \theta+\tau^4+\tau^5 & \tau^6 \\ \tau-\tau^4 & \theta+\tau^4-\tau^5 \end{pmatrix} \begin{pmatrix}
        -\tau & 1-s_1\tau \\ 1 & s_1^{1/q}
    \end{pmatrix}    \\
&= \begin{pmatrix}
    \theta & 0 \\ 0 & \theta
\end{pmatrix} +\begin{pmatrix}
    s_1^{1/q} & s_1^{1+(1/q)} \\ -1 & -s_1
\end{pmatrix}(\theta-\theta^q)\tau + \begin{pmatrix}
    -1 & s_1^{1/q}+s_1^q \\ 0 & 1
\end{pmatrix} \tau^2\\
&\hspace{2.8cm}+ \begin{pmatrix}
    s_1^{1/q} & s_1^{q^2+1/q} \\ -1 & -s_1^{q^2}
\end{pmatrix}\tau^3 + \begin{pmatrix}
    1 & s_1^{q^3}-1-s_1^{1/q} \\ 0 & 1
\end{pmatrix} \tau^4
\end{align*}

Since we chose $s_1$ to be a root of the polynomial $x^{q^4}-x-1$ the top coefficient matrix is even the identity matrix. However, for the top coefficient matrix being invertible, we could have chosen any $s_1$.
\end{exmp}

\begin{exmp}\label{ex:Alexis-example}
    Consider the Anderson $t$-module $\phi$ over $K=\mathbb{F}_q(\theta)$ given by
    \[
        \phi_t=\begin{pmatrix}
            \theta+\tau & \\
            \theta\tau^2 & \theta+\tau
        \end{pmatrix}.
    \]
    This is the example closely related to the $2$-dimensional construction given in \cite{al:paspatm}. By performing the algorithm for obtaining the elementary divisors, we get
    \[
        t\mathbb{I}_2-\phi_t\sim\begin{pmatrix}
            1 & \\
             & (t-\tau-\theta)(\sigma^2\theta^{-1})(t-\tau-\theta)
        \end{pmatrix}.
    \]
    It follows by computing the Newton polygon of $(t-\tau-\theta)(\sigma^2\theta^{-1})(t-\tau-\theta)$ in $\bK\sls{\sigma}[t]$ and by \cite[Thm.~7.2]{am:aefam} that $\phi_t$ defines a pure $t$-module of weight $1$.

    Let 
    \[  
        \vect{b} = \begin{pmatrix}
            \theta\tau^2 & \theta+\tau \\
             & \tau 
        \end{pmatrix}\vect{\kappa}.
    \]
    Then
    \begin{align*}
        t\vect{b} &= \begin{pmatrix}
            \theta\tau^2 & \theta+\tau \\
             & \tau 
        \end{pmatrix} \cdot \phi_t\cdot \begin{pmatrix}
            \theta\tau^2 & \theta+\tau \\
             & \tau 
        \end{pmatrix}^{-1} \vect{b}\\
        &=\begin{pmatrix}
            (\theta^q+\theta)+(\theta^{1-q}+1)\tau & -\theta^2\sigma+(\theta^{q^2}-\theta^q-\theta)-\theta^{1-q}\tau\\
            \tau & \theta^{q^2}-\theta^q
        \end{pmatrix}.
    \end{align*}
    Hence, the $K\sps{\sigma}$-lattice $\Lambda$ in $\mothat$ spanned by $\vect{b}$ satisfies
    \begin{equation*}
         t\Lambda = \tau \Lambda.
    \end{equation*}

    Now we decompose the matrix
    \[
        C:=\begin{pmatrix}
            \theta\tau^2 & \theta+\tau \\
             & \tau 
        \end{pmatrix}=A\cdot D\cdot B.
    \]
    In particular, we obtain
    \begin{align*}
        A&=\begin{pmatrix}
            1 & (\theta+\tau)\sigma\\
             & 1
        \end{pmatrix}\begin{pmatrix}
            \theta & \\
             & 1
        \end{pmatrix}\begin{pmatrix}
             & 1\\
            1 & 
        \end{pmatrix}\\
        D&=\begin{pmatrix}
            \tau & \\
             & \tau^2
        \end{pmatrix}\\
        B&=\begin{pmatrix}
             & 1\\
            1 & 
        \end{pmatrix}.
    \end{align*}
    Then we have
    \[ 
        \vect{b''}=A^{-1}\vect{b}=\begin{pmatrix}
             & 1 \\
            \theta^{-1} & -\theta^{-1}(\theta+\tau)\sigma
        \end{pmatrix}\cdot \begin{pmatrix}
            \theta\tau^2 & \theta+\tau \\ 
             & \tau
        \end{pmatrix}\vect{\kappa}= \begin{pmatrix}
             & \tau \\ 
            \tau^2 & 
        \end{pmatrix} \vect{\kappa},
    \]
    and the matrix $\Theta\in \GL_d(\bK\sls{\sigma})$ satisfying
    $t \vect{b''}=\Theta \vect{b''}$ is
    \[ 
        \Theta = \begin{pmatrix}
            \theta^q+\tau & \theta^q\tau \\
             & \theta^{q^2}+\tau
        \end{pmatrix}
    \]
    with $\ord_{\sigma}(\Theta)=-1$, and
    \[ 
        c_{-1}(\Theta)= \begin{pmatrix}
            1 & \theta^q \\
             & 1
        \end{pmatrix}.
    \]
    If we choose $s_0,s_1\in\bK$ with
    \[
        \begin{cases}
            s_0^q-s_0-\theta^q=0\\
            s_1^q-s_1-s_0(\theta^{q^2}-\theta^q)=0
        \end{cases},
    \]
    then the matrix
    \[
        G:=\begin{pmatrix}
            1 & s_0+s_1\sigma\\
             & 1
        \end{pmatrix}
    \]
    has the property that $H:=G\Theta G^{-1}$ satisfies
    \[
        \mathrm{ord}_\sigma(H_{12})>-1+2-1=0.
    \]
    So the matrix $\tilde{G}$ is
    \[  
        \tilde{G}=D^{-1}\begin{pmatrix}
            1 & s_0+s_1\sigma \\
            0 & 1
        \end{pmatrix}D = \begin{pmatrix}
            1 & s_0^{1/q}\tau+s_1^{1/q} \\ 0 & 1
        \end{pmatrix},
    \]
    and the desired basis $\vect{\kappa'}$ is given by
    \begin{align*}
        \vect{\kappa'} &= \tilde{G}\vect{\kappa''}=\tilde{G}B\vect{\kappa} \\
        &= \begin{pmatrix}
            s_0^{1/q}\tau+s_1^{1/q} & 1 \\
            1 & 
        \end{pmatrix}\vect{\kappa}
    \end{align*}
    with $t$-action
    $t\cdot \vect{\kappa'}=\tilde{H}\vect{\kappa'}$ where
    \[
        \tilde{H}=\begin{pmatrix}
            \theta+\tau & (s_0^{1/q}\theta^q+s_1^{1/q}-\theta s_0^{1/q}-s_1)\tau\\
             & \theta+\tau
        \end{pmatrix}=\begin{pmatrix}
            \theta+\tau & \\
             & \theta+\tau
        \end{pmatrix} .
    \]
\end{exmp}

\def\cprime{$'$}

\end{document}